\documentclass{amsart}
\usepackage{amsmath,amsfonts,amsthm}
\usepackage{amssymb,latexsym}
\usepackage{graphics}
\usepackage[colorlinks]{hyperref}

\usepackage[all]{xypic}
\theoremstyle{plain}
\theoremstyle{definition}
\newtheorem{theorem}{Theorem}[section]
\newtheorem{lemma}[theorem]{Lemma}

\newtheorem{corollary}[theorem]{Corollary}

\newtheorem{note}[theorem]{Note}

\theoremstyle{remark}

\numberwithin{equation}{section}
\newcommand{\SP}{\: \: \: \: \:}

\title{Possible heights of Alexandroff square transformation groups}
\author[F. Ayatollah Zadeh Shirazi, F. Ebrahimifar, R. Yaghmaeian, H.Yahyaoghli]{Fatemah Ayatollah Zadeh Shirazi, \\
Fatemeh Ebrahimifar, \\ 
Reza Yaghmaeian, \\
Hamed Yahyaoghli}
\begin{document}
%%%%%%%%%%%%%%%%%%%%%%%%%%%%%%%%%%%% abstract
\begin{abstract} 
In the following text we compute possible heights of $\mathbb A$
(Alexandroff square), $\mathbb O$ (unit square $[0,1]\times[0,1]$
with lexicographic order topology) and $\mathbb U$ (unit square
$[0,1]\times[0,1]$ with induced topology of Euclidean plane). We
prove $P_h(\mathbb{A})=\{n:n\geq5\}\cup\{+\infty\}$, $P_h(\mathbb{O})=\{n:n\geq4\}\cup\{+\infty\}$, $P_h(\mathbb{U})=\{n:n\geq1\}\cup\{+\infty\}$ (where for
topological space $X$, by $P_h(X)$ we mean the collection of
heights of transformation groups with phase space $X$.
In this way we also prove that there is not any topological transitive
(resp. Devaney chaotic) Alexandroff square transformation group.
\end{abstract}
\maketitle
%%%%%%%%%%%%%%%%%%%%%%%%%%%%%%%%%%%% MSC
\noindent {\small {\bf 2010 Mathematics Subject Classification:}  54H15, 54H20 \\
{\bf Keywords:}} Alexandroff square, height, orbit space, transformation group.
%%%%%%%%%%%%%%%%%%%%%%%%%%%%%%%%%%%%
\section{Introduction}
\noindent Studying closed unit ball $\{<x,y>\in\mathbb{R}^2:
x^2+y^2\leq1\}$ with induced topology of Euclidean plane $\mathbb{R}^2$
is one of the main purposes of so many text (old and new) (see e.g.,
\cite{rudin, zhu}.
Let's mention that unit disk and unit square $[0,1]\times[0,1]$
with induced topology of Euclidean plane, are homeomorph. 
\\
On the other hand many texts deal with dynamical properties 
of special topological spaces \cite{fort, graph}. 
In the following text we have a comparative study on dynamical properties of unit square 
transformation groups with emphasize on their heights (and obit spaces), where unit square $[0,1]\times[0,1]$ equipped with Euclidean topology,
lexicographic order topology, Alexandroff square topology. For more convenient suppose (by $<x,y>$ we mean the 
ordered pair $\{x,\{x,y\}\}$):
\begin{itemize}
\item $\mathbb{A}$ is $[0,1]\times[0,1]$ as Alexandroff square.
\item $\mathbb{O}$ is $[0,1]\times[0,1]$ equipped with lexicographic order topology,
\item $\mathbb{U}$ is $[0,1]\times[0,1]$ equipped with Euclidean plane ${\mathbb R}^2$,
\end{itemize}
Where for $<x,y>,<s,t>\in[0,1]\times[0,1]$ we define lexicographic order $\preceq_\ell$ with
$<x,y>\preceq_\ell<s,t>$ if and only if ``$x<s$'' or ``$x=s$ and $y\leq t$''.
Alexandroff square $\mathbb{A}=[0,1]\times[0,1]$
equipped with topological basis generated by~\cite{counter}:
\\
$\bullet$ $\{x\}\times U$ where $x\in[0,1]$ and $U$ is an open subset of $[0,1]$ (with
induced topology of Euclidean line $\mathbb R$) and $x\notin U$, 
\\
$\bullet$ $([0,1]\times U)\setminus(\{x_1,\ldots,x_n\}\times[0,1])$ where $U$ is an open subset of $[0,1]$ (with
induced topology of Euclidean line $\mathbb R$).
\\
As it has been mentioned in~\cite{counter}, $\mathbb{A}$ and $\mathbb O$ are compact Hausdorff non--metrizable spaces.
Consider the following notations and sets 
(for $x,y\in\mathbb{R}$ let $(x,y)=\{z\in\mathbb{R}:x<z<y\}$):

$\Delta:=\{<x,x>:x\in[0,1]\}$;

$\mathsf{P}_1:=<0,0>, \mathsf{P}_2:=<0,1>, \mathsf{P}_3:=<1,1>, \mathsf{P}_4:=<1,0>$;

$\mathsf{L}_1:=\{0\}\times(0,1), \mathsf{L}_2:=(0,1)\times\{1\}, \mathsf{L}_3:=\{1\}\times(0,1), \mathsf{L}_4:=(0,1)\times\{0\}$.
\\
$\:$
%%%%%%%%%%%%%%%%%%%%%%%%%%%%%%%%%%%%
\begin{center} {\tiny
% This is a LaTeX picture output by TeXCAD.
% File name: [1.pic].
% Version of TeXCAD: 4.3
% Reference / build: 30-Jun-2012 (rev. 105)
% For new versions, check: http://texcad.sf.net/
% Options on the following lines.
%\grade{\on}
%\emlines{\off}
%\epic{\off}
%\beziermacro{\on}
%\reduce{\on}
%\snapping{\off}
%\pvinsert{% Your \input, \def, etc. here}
%\quality{8.000}
%\graddiff{0.005}
%\snapasp{1}
%\zoom{4.0000}
\unitlength 0.4mm % = 2.845pt
\linethickness{0.4pt}
\ifx\plotpoint\undefined\newsavebox{\plotpoint}\fi % GNUPLOT compatibility
\begin{picture}(101.75,99.25)(0,0)
\put(7.75,5.5){\framebox(90,90)[]{}}
\put(99.75,21.5){\line(0,1){0}}
\put(98,5.5){\circle*{2.5}}
\put(7.75,6){\circle*{2.5}}
\put(8.25,95.25){\circle*{2.5}}
\put(97.75,95.25){\circle*{2.5}}
\put(97.25,95.75){\line(-1,-1){89.75}}
\put(7.5,6){\line(0,1){0}}
\put(47.25,52){\makebox(0,0)[cc]{$\Delta$}}
\put(4.75,4){\makebox(0,0)[cc]{$\mathsf{P}_1$}}
\put(3.75,97.25){\makebox(0,0)[cc]{$\mathsf{P}_2$}}
\put(100.75,97.75){\makebox(0,0)[cc]{$\:\:\mathsf{P}_3$}}
\put(101.75,4.75){\makebox(0,0)[cc]{$\:\:\mathsf{P}_4$}}
\put(1,48.25){\makebox(0,0)[cc]{$\mathsf{L}_1$}}
\put(51.75,99.25){\makebox(0,0)[cc]{$\mathsf{L}_2$}}
\put(101.5,54){\makebox(0,0)[cc]{$\:\:\mathsf{L}_3$}}
\put(48.5,2.75){\makebox(0,0)[cc]{$\mathsf{L}_4$}}
\end{picture}
}
\\
Fig. 1
\end{center}
%%%%%%%%%%%%%%%%%%%%%%%%%%%%%%%%%%%%
%%%%%%%%%%%%%%%%%%%%%%%%%%%%%%%%%%%%
\subsection*{Background on transformation groups}
By a (topological) transformation group $(G,X,\rho)$
or simply $(G,X)$ we mean a compact Hausdorff topological
space $X$ (phase space), discrete topological group $G$
(phase group) with identity $e$ and continuous map
$\rho:G\times X\to X$, $\rho(g,x)=gx$ ($g\in G,x\in X$) such
that for all $x\in X$ and $g,h\in G$ we have $ex=x$ and $g(hx)=(gh)x$.
note to the fact that for all $g\in G$, $\rho_g:X\to X$ with 
$\rho_g(x)=gx$  is a homeomorphism on $X$, and 
$\rho_g\rho_h=\rho_{gh}$ we may suppose $G$ is a group of 
self--homeomorphisms of $X$ under the operation of composition
of maps. In transformation group $(G,X)$ for $x\in X$
we call $Gx:=\{gx:g\in G\}$ the orbit of $x$ (under $G$) and
$\frac{X}{G}:=\{Gy:y\in X\}$ the orbit space of $(G,X)$, also
we say nonempty subset $D$ of $X$ is invariant ($G-$invariant)
if $GD:=\{gy:g\in G,y\in D\}\subseteq D$,
for more details on transformation groups (and orbit spaces)
see~\cite{ellis, kuwa}. In topological space $X$ suppose $\mathcal{G}_X$ is the collection of all homeomorphisms
$h:X\to X$. 
\\
Closed and open invariant subsets of a transformation group $(G,X)$
have important role in studying its dynamical properties
(see e.g. \cite{vesnik} in transitivity in transformation groups). In transformation group $(G,X)$ let $h(G,X):=\sup\{n\geq0:$ there exist closed invariant 
subsets $D_0,\ldots,D_n$ of $X$ with $\varnothing\neq D_0\varsubsetneq D_1\varsubsetneq\cdots\varsubsetneq D_n=X\}$,
i,e., $h(G,X)=+\infty$ if $\{\overline{Gx}:x\in X\}$ is infinite and $h(G,X)=card(\{\overline{Gx}:x\in X\})-1$ otherwise~\cite{golestani}.
Also we call $P_h(X):=\{h(G,X):G$ is a subgroup of $\mathcal{G}_X\}$ the collection of all possible heights of $X$. In transformation
group $(G,X)$ the map $\varphi:\{\overline{Gy}:y\in X\}\to
\{\overline{\mathcal{G}_Xy}:y\in X\}$ with $\varphi(\overline{Gy})=
\overline{\mathcal{G}_Xy}$ (for $y\in X$) is onto, so
$h(\mathcal{G}_X,X)\leq h(G,X)$ therefore
$\min P_h(X)=h(\mathcal{G}_X,X)$.
%%%%%%%%%%%%%%%%%%%%%%%%%%%%%%%%%%%%
%%%%%%%%%%%%%%%%%%%%%%%%%%%%%%%%%%%%
\section{Computing $\dfrac{\mathbb A}{\mathcal{G}_{\mathbb A}}$,
$\dfrac{\mathbb O}{\mathcal{G}_{\mathbb O}}$ and
$\dfrac{\mathbb U}{\mathcal{G}_{\mathbb U}}$}
\noindent By considering the definition of height of transformation group $(G,X)$ it's evident that for computing $h(G,X)$ one may compute $\{\overline{Gy}:y\in X\}$, so one may begin with computing $\frac{X}{G}=\{Gy:y\in X\}$. Hence by $\min P_h(X)=h(\mathcal{G}_X,X)$, for obtaining $P_h(X)$ one may pay attention to  $h(\mathcal{G}_X,X)$
and finding out $\frac{X}{\mathcal{G}_X}$. 
In this section we pay attention to $\frac{X}{\mathcal{G}_X}$ 
for $X=\mathbb{U},\mathbb{O},\mathbb{A}$.
%%%%%%%%%%%%%%%%%%%%%%%%%%%%%%%%%%%%
\begin{lemma}\label{salam10}
For homeomorphism $\mathfrak{a}:\mathbb{A}\to\mathbb{A}$ we have:
\\
1. $\mathfrak{a}(\{\mathsf{P}_1,\mathsf{P}_3\})=\{\mathsf{P}_1,\mathsf{P}_3\}$ and $\mathfrak{a}(\Delta)=\Delta$;
\\
2. $\mathfrak{a}(\mathsf{L}_2\cup\mathsf{L}_4\cup\{\mathsf{P}_2,\mathsf{P}_4\})=\mathsf{L}_2\cup\mathsf{L}_4\cup\{\mathsf{P}_2,\mathsf{P}_4\}$,
\\
3. for all $s\in[0,1]$ there exists $t\in[0,1]$ with $\mathfrak{a}(\{s\}\times[0,1])=\{t\}\times[0,1]$ and $\mathfrak{a}\{<s,0>,<s,1>\}=\{<t,0>,<t,1>\}$;
\\
4. exactly one the following cases occurs:
	\begin{itemize}
	\item[a.] $\mathfrak{a}(\mathsf{P}_i)=\mathsf{P}_i$ for $i=1,2,3,4$,  $\mathfrak{a}(\mathsf{L}_1)=\mathsf{L}_1$ and  $\mathfrak{a}(\mathsf{L}_3)=\mathsf{L}_3$;
	\item[b.] $\mathfrak{a}(\mathsf{P}_1)=\mathsf{P}_3, \mathfrak{a}(\mathsf{P}_2)=\mathsf{P}_4, \mathfrak{a}(\mathsf{P}_3)=\mathsf{P}_1, \mathfrak{a}(\mathsf{P}_4)=\mathsf{P}_2$,
	$\mathfrak{a}(\mathsf{L}_1)=\mathsf{L}_3$ and $ \mathfrak{a}(\mathsf{L}_3)=\mathsf{L}_1$.
	\end{itemize}
\end{lemma}
%%%%%%%%%%%%%%%%%%%%%%%%%%%%%%%%%%%%
\begin{proof}
{\bf 1.} Using 
	the fact that $\mathbb{A}$ has a local countable topological basis on $\mathfrak{x}\in\mathbb{A}$ if and only if 
	$\mathfrak{x}\in\mathbb{A}\setminus\Delta$, we have $\mathfrak{a}(\Delta)=\Delta$.
	Note that subspace topology on $\Delta$ induced by $\mathbb{A}$
	coincides with subspace topology on $\Delta$ induced by $\mathbb{U}$
	hence $\mathfrak{a}(\{\mathsf{P}_1,\mathsf{P}_3\})=\{\mathsf{P}_1,\mathsf{P}_3\}$.
\\
{\bf 2.} $\mathbb{A}$ has a local countable topological basis at $\mathfrak{x}\in\mathbb{A}$ like $\{B_n:n\geq1\}$ such that all elements of
	$\{B_n\setminus\{\mathfrak{x}\}:n\geq1\}$ are connected if and only if 
	$\mathfrak{x}\in\mathsf{L}_2\cup\mathsf{L}_4\cup\{\mathsf{P}_2,\mathsf{P}_4\}$.
\\
{\bf 3.} Consider $s\in[0,1]$, using (1) and (2) we have 
	$<a,b>:=\mathfrak{a}<s,0>,<c,d>:=\mathfrak{a}<s,1>\in\mathsf{L}_2\cup\mathsf{L}_4\cup\{\mathsf{P}_i:1\leq i\leq4\}$.
	So $b,d\in\{0,1\}$. Choose $x\in[0,1]$ and suppose $<u,v>:=\mathfrak{a}<s,x>$. By item (1), 
	$\mathfrak{a}(\{s\}\times[0,1])\cap\Delta=\mathfrak{a}<s,s>=:<t,t>$ if $u\neq t$ then $<u,u>\notin\mathfrak{a}(\{s\}\times[0,1])$, let:
	\begin{eqnarray*}
	U&:= & (\{u\}\times([0,1]\setminus\{u\}))\cap \mathfrak{a}(\{s\}\times[0,1])=(\{u\}\times[0,1])\cap \mathfrak{a}(\{s\}\times[0,1])\:,  \\ 
	V& := & (\mathbb{A}\setminus(\{u\}\times[0,1]))\cap \mathfrak{a}(\{s\}\times[0,1])\:, 
	\end{eqnarray*}
	then $U,V$ is a separation of $\mathfrak{a}(\{s\}\times[0,1])$ ($\mathfrak{a}<s,x>\in U$ and $\mathfrak{a}<s,s>\in V$)
	which is in contradiction with connectedness of $\mathfrak{a}(\{s\}\times[0,1])$. Thus $u=t$ and 
	$\mathfrak{a}<s,x>\in\{t\}\times[0,1]$ for all $x\in[0,1]$., so $\mathfrak{a}(\{s\}\times[0,1])\subseteq\{t\}\times[0,1]$.
	In particular $a=c=t$, so $<t,b>=\mathfrak{a}<s,0>,<t,d>=\mathfrak{a}<s,1>$ with $b,d\in\{0,1\}$
	(since $\mathfrak{a}<s,0>\neq\mathfrak{a}<s,1>$ we have $b\neq d$). Thus
	$\mathfrak{a}\restriction_{\{s\}\times[0,1]}:\{s\}\times[0,1]\to\{t\}\times[0,1]$ is a continuous map
	with $<t,0>,<t,1>\in\mathfrak{a}(\{s\}\times[0,1])$ which completes the proof.
\\
{\bf 4.} First suppose $\mathfrak{a}(\mathsf{P}_1)=\mathsf{P}_1$, then by (1), $\mathfrak{a}(\mathsf{P}_3)=\mathsf{P}_3$,
so by (3) we have $\mathfrak{a}(\mathsf{L}_1)=\mathsf{L}_1$, $\mathfrak{a}(\mathsf{L}_3)=\mathsf{L}_3$,
$\mathfrak{a}(\mathsf{P}_2)=\mathsf{P}_2$ and $\mathfrak{a}(\mathsf{P}_4)=\mathsf{P}_4$.
\\
Now suppose $\mathfrak{a}(\mathsf{P}_1)\neq\mathsf{P}_1$, then by (1), $\mathfrak{a}(\mathsf{P}_1)=\mathsf{P}_3$
and $\mathfrak{a}(\mathsf{P}_3)=\mathsf{P}_1$
so by (3) we have $\mathfrak{a}(\mathsf{L}_1)=\mathsf{L}_3$, $\mathfrak{a}(\mathsf{L}_3)=\mathsf{L}_1$,
$\mathfrak{a}(\mathsf{P}_2)=\mathsf{P}_4$ and $\mathfrak{a}(\mathsf{P}_4)=\mathsf{P}_2$.
\end{proof}
%%%%%%%%%%%%%%%%%%%%%%%%%%%%%%%%%%%%
\begin{lemma}\label{salam20}
For homeomorphism $\mathfrak{o}:\mathbb{O}\to\mathbb{O}$ we have:
\\
1. $\mathfrak{o}:\mathbb{O}\to\mathbb{O}$ is order preserving or anti--order preserving;
\\
2. $\mathfrak{o}(\{\mathsf{P}_1,\mathsf{P}_3\})=\{\mathsf{P}_1,\mathsf{P}_3\}$;
\\
3. $\mathfrak{o}(\mathsf{L}_2\cup\mathsf{L}_4\cup\{\mathsf{P}_2,\mathsf{P}_4\})=\mathsf{L}_2\cup\mathsf{L}_4\cup\{\mathsf{P}_2,\mathsf{P}_4\}$,
\\
4. for all $s\in[0,1]$ there exists $t\in[0,1]$ with $\mathfrak{o}(\{s\}\times[0,1])=\{t\}\times[0,1]$ and $\mathfrak{o}\{<s,0>,<s,1>\}=\{<t,0>,<t,1>\}$;
\\
5. exactly one the following cases occurs:
	\begin{itemize}
	\item[a.] $\mathfrak{o}(\mathsf{P}_i)=\mathsf{P}_i$, $\mathfrak{o}(\mathsf{L}_i)=\mathsf{L}_i$ for $i=1,2,3,4$ and $\mathfrak{o}:\mathbb{O}\to\mathbb{O}$ is order preserving;
	\item[b.] $\mathfrak{o}(\mathsf{P}_1)=\mathsf{P}_3, \mathfrak{o}(\mathsf{P}_2)=\mathsf{P}_4, \mathfrak{o}(\mathsf{P}_3)=\mathsf{P}_1, \mathfrak{o}(\mathsf{P}_4)=\mathsf{P}_2$,
	$\mathfrak{o}(\mathsf{L}_1)=\mathsf{L}_3$, $\mathfrak{o}(\mathsf{L}_2)=\mathsf{L}_4$, $ \mathfrak{o}(\mathsf{L}_3)=\mathsf{L}_1$, $\mathfrak{o}(\mathsf{L}_4)=\mathsf{L}_2$ and $\mathfrak{o}:\mathbb{O}\to\mathbb{O}$ is anti--order preserving.
	\end{itemize}
\end{lemma}
%%%%%%%%%%%%%%%%%%%%%%%%%%%%%%%%%%%%
\begin{proof}
{\bf 2.} Use (1) and $\mathsf{P}_1=\max\mathbb{O}$, $\mathsf{P}_3=\min\mathbb{O}$.
\\
{\bf 3.} Use the fact that all open neighbourhoods of $\mathfrak{x}\in\mathbb{O}$ is non--metrizable if and only if 
	$\mathfrak{x}\in\mathsf{L}_2\cup\mathsf{L}_4\cup\{\mathsf{P}_2,\mathsf{P}_4\}$.
\\
{\bf 4.}  Consider $s\in[0,1]$, using (2) and (3) we have 
	$<a,b>:=\mathfrak{o}<s,0>,<c,d>:=\mathfrak{o}<s,1>\in\mathsf{L}_2\cup\mathsf{L}_4\cup\{\mathsf{P}_i:1\leq i\leq4\}$.
	So $b,d\in\{0,1\}$. Choose $x\in[0,1]$ and suppose $<t,v>:=\mathfrak{o}<s,x>$. 
	If $t\neq a$ then we may choose $r\in\{\frac{a+t}{2},\frac{a+2t}{3},\frac{a+3t}{4},\frac{a+4t}{5}\}\setminus\{a,c,t\}$.
	Then $<r,0>\notin\mathfrak{o}(\{s\}\times[0,1])$ and for:
	\begin{eqnarray*}
	U & := & \{<z,w>\in\mathbb{O}:<z,w>\prec_\ell<r,0>\}\cap\mathfrak{o}(\{s\}\times[0,1])\:, \\
	V & := & \{<z,w>\in\mathbb{O}:<r,0>\prec_\ell<z,w>\}\cap\mathfrak{o}(\{s\}\times[0,1])\:,
	\end{eqnarray*}
	$U,V$ is a separation of $\mathfrak{o}(\{s\}\times[0,1])$ 
	which is in contradiction with connectedness of $\mathfrak{o}(\{s\}\times[0,1])$. Thus $t=a$ and 
	$\mathfrak{o}(\{s\}\times[0,1])\subseteq\{t\}\times[0,1]$.
	In particular $a=c=t$, so $<t,b>=\mathfrak{o}<s,0>,<t,d>=\mathfrak{o}<s,1>$ with $b,d\in\{0,1\}$. Thus
	$\mathfrak{o}\restriction_{\{s\}\times[0,1]}:\{s\}\times[0,1]\to\{t\}\times[0,1]$ is a continuous map
	with $<t,0>,<t,1>\in\mathfrak{o}(\{s\}\times[0,1])$ which completes the proof.
\\
{\bf 5.} (a) Suppose $\mathfrak{o}:\mathbb{O}\to\mathbb{O}$ is order preserving. So 
$\mathfrak{o}(\mathsf{P}_1)=\mathfrak{o}(\min\mathbb{O})=
\min\mathbb{O}=\mathsf{P}_1$ and $\mathfrak{o}(\mathsf{P}_3)=\mathfrak{o}(\max\mathbb{O})=
\max\mathbb{O}=\mathsf{P}_3$, also by (3) we have 
\[\mathfrak{o}(\mathsf{P}_2)=\mathfrak{o}(\min(\mathsf{L}_2\cup\mathsf{L}_4\cup\{\mathsf{P}_2,\mathsf{P}_4\}))=\min(\mathsf{L}_2\cup\mathsf{L}_4\cup\{\mathsf{P}_2,\mathsf{P}_4\})=\mathsf{P}_2\]
and 
\[\mathfrak{o}(\mathsf{P}_4)=\mathfrak{o}(\max(\mathsf{L}_2\cup\mathsf{L}_4\cup\{\mathsf{P}_2,\mathsf{P}_4\}))=\max(\mathsf{L}_2\cup\mathsf{L}_4\cup\{\mathsf{P}_2,\mathsf{P}_4\})=\mathsf{P}_4\:.\]
Hence by (4) we have $\mathfrak{o}(\mathsf{L}_1)=\mathsf{L}_1$ and $\mathfrak{o}(\mathsf{L}_4)=\mathsf{L}_4$. Consider
$s\in[0,1]$, by (4) there exists $t\in[0,1]$ with $\mathfrak{o}(\{s\}\times[0,1])=\{t\}\times[0,1]$ so
\[\mathfrak{o}<s,0>=\mathfrak{o}(\min(\{s\}\times[0,1]))=\min\mathfrak{o}(\{s\}\times[0,1])=\min(\{t\}\times[0,1])=<t,0>\:,\]
which shows $\mathfrak{o}(\mathsf{L}_1\cup\{\mathsf{P}_1,\mathsf{P}_2\})\subseteq\mathsf{L}_1\cup\{\mathsf{P}_1,\mathsf{P}_2\}$
and $\mathfrak{o}(\mathsf{L}_4)\subseteq\mathsf{L}_4$; also by a similar method we have $\mathfrak{o}<s,1>=<t,1>$ which leads also to 
$\mathfrak{o}(\mathsf{L}_2)\subseteq\mathsf{L}_2$. Use (2) to obtain $\mathfrak{o}(\mathsf{L}_2)=\mathsf{L}_2$
and $\mathfrak{o}(\mathsf{L}_4)=\mathsf{L}_4$.
\\
(b) Use a similar method described for (a).
\end{proof}
%%%%%%%%%%%%%%%%%%%%%%%%%%%%%%%%%%%%
\begin{theorem}\label{salam25}
$\mathfrak{o}:\mathbb{O}\to\mathbb{O}$ is an order preserving homeomorphism if and only if there exist
order preserving homeomorphism $\theta:[0,1]\to[0,1]$  and $\mathop{\mu:[0,1]\to[0,1]^{[0,1]}}\limits_{t\mapsto\mu_t}$
such that for all $t\in[0,1]$, $\mu_t:[0,1]\to[0,1]$ is an order preserving homeomorphism and
$\mathfrak{o}<s,t>=<\theta(s),\mu_s(t)>$.
\\
Also $\mathfrak{o}:\mathbb{O}\to\mathbb{O}$ is an anti--order preserving homeomorphism if and only if there exist
anti--order preserving homeomorphism $\theta:[0,1]\to[0,1]$  and $\mathop{\mu:[0,1]\to[0,1]^{[0,1]}}\limits_{t\mapsto\mu_t}$
such that for all $t\in[0,1]$, $\mu_t:[0,1]\to[0,1]$ is an anti--order preserving homeomorphism and
$\mathfrak{o}<s,t>=<\theta(s),\mu_s(t)>$.
\end{theorem}
%%%%%%%%%%%%%%%%%%%%%%%%%%%%%%%%%%%%
\begin{proof}
First suppose $\mathfrak{o}:\mathbb{O}\to\mathbb{O}$ is an order preserving homeomorphism, by 
\linebreak
Lemma~\ref{salam20}
for each $s\in[0,1]$ there exists $t\in[0,1]$ with $\mathfrak{o}(\{s\}\times[0,1])=\{t\}\times[0,1]$
let $\theta(s):=t$. Also by Lemma~\ref{salam20} (since $\mathfrak{o}\restriction_{\mathsf{L}_2\cup\{\mathsf{P}_2,\mathsf{P}_3\}}:
\mathsf{L}_2\cup\{\mathsf{P}_2,\mathsf{P}_3\}\to\mathsf{L}_2\cup\{\mathsf{P}_2,\mathsf{P}_3\}$ is order preserving and bijection),
$\theta:[0,1]\to[0,1]$ is order preserving and bijection, thus it is an order preserving homeomorphism on $[0,1]$.
Now for $s\in[0,1]$, considering homeomorphism
$\mathfrak{o}\restriction_{\{s\}\times[0,1]}:\{s\}\times[0,1]\to\{\theta(s)\}\times[0,1]$, we may define
homeomorphism $\mu_s:[0,1]\to[0,1]$ with $\mathfrak{o}<s,t>=<\theta(s),\mu_s(t)>$. For $x,y\in[0,1]$ with $x\leq y$
since $<s,x>\preceq_\ell<s,y>$ we have 
\[<\theta(s),\mu_s(x)>=\mathfrak{o}<s,x>\preceq_\ell\mathfrak{o}<s,y>=<\theta(s),\mu_s(y)>\]
which leads to $\mu_s(x)\leq\mu_s(y)$ and  $\mu_s:[0,1]\to[0,1]$ is order preserving too.
\\
Conversely, consider order preserving homeomorphism $\theta:[0,1]\to[0,1]$  and 
\linebreak
$\mathop{\mu:[0,1]\to[0,1]^{[0,1]}}\limits_{t\mapsto\mu_t}$
such that for all $t\in[0,1]$, $\mu_t:[0,1]\to[0,1]$ is an order preserving homeomorphism and define $\mathfrak{o}:\mathbb{O}\to\mathbb{O}$ 
with $\mathfrak{o}<s,t>=<\theta(s),\mu_s(t)>$. It's clear that $\mathfrak{o}:\mathbb{O}\to\mathbb{O}$  is order preserving and bijective
which leads to continuity of $\mathfrak{o}:\mathbb{O}\to\mathbb{O}$ under order topology.
\\
In order to complete the proof consider homeomorphism $\mathop{\varphi:\mathbb{O}\to\mathbb{O}}\limits_{<s,t>\mapsto<1-s,1-t>}$
and note that $\mathfrak{o}:\mathbb{O}\to\mathbb{O}$ is an anti--order preserving homeomorphism if and only if
$\varphi\circ\mathfrak{o}:\mathbb{O}\to\mathbb{O}$ is an order preserving homeomorphism.
\end{proof}
%%%%%%%%%%%%%%%%%%%%%%%%%%%%%%%%%%%%
\begin{note}\label{salam27}
If $\mathfrak{a}:\mathbb{A}\to\mathbb{A}$ is a homeomorphism, then exist
homeomorphism $\theta:[0,1]\to[0,1]$ with $\theta(\{0,1\})=\{0,1\}$  and $\mathop{\mu:[0,1]\to[0,1]^{[0,1]}}\limits_{t\mapsto\mu_t}$
such that for all $t\in[0,1]$, $\mu_t:[0,1]\to[0,1]$ is a homeomorphism with $\mu_t(t)=\theta(t)$ and
$\mathfrak{a}<s,t>=<\theta(s),\mu_s(t)>$ (note that $\mathfrak{a}\restriction_{\Delta\cup\{\mathsf{P}_1,\mathsf{P}_3\}}:\Delta\cup\{\mathsf{P}_1,\mathsf{P}_3\}\to\Delta\cup\{\mathsf{P}_1,\mathsf{P}_3\}$ is a homeomorphism).
\end{note}
%%%%%%%%%%%%%%%%%%%%%%%%%%%%%%%%%%%%
\begin{corollary}
For homeomorphisms $p,q:[0,1]\to[0,1]$,
consider 
\[\mathop{p\times q:[0,1]\times[0,1]\to[0,1]\times[0,1]}\limits_{<s,t>\mapsto<p(s),q(t)>}\:,\]
then we have:
\begin{itemize}
\item[1.] $p\times q:\mathbb{A}\to\mathbb{A}$ is a homeomorphism if and only if $p=q$;
\item[2.] $p\times q:\mathbb{O}\to\mathbb{O}$ is a homeomorphism if and only if $p\circ q:[0,1]\to[0,1]$ is order preserving;
\item[3.] $p\times q:\mathbb{U}\to\mathbb{U}$ is a homeomorphism.
\end{itemize}
\end{corollary}
%%%%%%%%%%%%%%%%%%%%%%%%%%%%%%%%%%%%
\begin{proof}
{\bf 1.} If $p\times q:\mathbb{A}\to\mathbb{A}$ is a homeomorphism, then by Lemma~\ref{salam10}
we have $p\times q(\Delta)=\Delta$, thus for all $t\in[0,1]$ we have $<p(t),q(t)>=p\times q(t,t)\in\Delta$
which shows $p(t)=q(t)$ and leads to $p=q$.
\\
{\bf 2.} Suppose $p\times q:\mathbb{O}\to\mathbb{O}$ is a homeomorphism, by Lemma~\ref{salam20} 
one of the following cases accours:
\begin{itemize}
\item $p\times q:\mathbb{O}\to\mathbb{O}$ is order preserving: in this case $p,q:[0,1]\to[0,1]$ are 
	order preserving too, thus $p\circ q:[0,1]\to[0,1]$ is order preserving;
\item $p\times q:\mathbb{O}\to\mathbb{O}$ is anti--order preserving: in this case $p,q:[0,1]\to[0,1]$ are 
	anti--order preserving too, thus $p\circ q:[0,1]\to[0,1]$ is order preserving.
\end{itemize}
Using the above cases $p\circ q:[0,1]\to[0,1]$ is order preserving.
\\
Conversely suppose $p\circ q:[0,1]\to[0,1]$ is order preserving, thus either 
 ``$p,q:[0,1]\to[0,1]$ are order preserving'' or  ``$p,q:[0,1]\to[0,1]$ are anti--order preserving''.
Use Theorem~\ref{salam25} to complete the proof of this item.
\end{proof}
%%%%%%%%%%%%%%%%%%%%%%%%%%%%%%%%%%%%
\begin{theorem}\label{salam30}
We have:
\begin{eqnarray*}
\frac{\mathbb{A}}{\mathcal{G}_{\mathbb{A}}} & = & \{\{\mathsf{P}_1,\mathsf{P}_3\},\{\mathsf{P}_2,\mathsf{P}_4\},\mathsf{L}_1\cup\mathsf{L}_3,\mathsf{L}_2\cup\mathsf{L}_4,\Delta\setminus\{\mathsf{P}_1,\mathsf{P}_3\},((0,1)\times(0,1))\setminus\Delta\}\:, \\
\frac{\mathbb{O}}{\mathcal{G}_{\mathbb{O}}} & = & \{\{\mathsf{P}_1,\mathsf{P}_3\},\{\mathsf{P}_2,\mathsf{P}_4\},\mathsf{L}_1\cup\mathsf{L}_3,\mathsf{L}_2\cup\mathsf{L}_4,(0,1)\times(0,1)\}\:, \\
\frac{\mathbb{U}}{\mathcal{G}_{\mathbb{U}}} & = & \{(0,1)\times(0,1),\mathbb{U}\setminus((0,1)\times(0,1))\}\:.
\end{eqnarray*}
\end{theorem}
%%%%%%%%%%%%%%%%%%%%%%%%%%%%%%%%%%%%
\begin{proof} We prove case by case. 
Note that $\varphi: X\to X$ with $\varphi<s,t>=<1-s,1-t>$ (for $(s,t)\in X$)
for $X=\mathbb{A},\mathbb{O},\mathbb{U}$ is homeomorphism. Also for $x,y\in(0,1)$ consider hommeoorphism
$f_{x,y}:[0,1]\to[0,1]$ with:
\[f_{x,y}(t)=\left\{\begin{array}{lc}  \dfrac{y}{x}t & 0\leq t\leq x \:, \\ & \\ \dfrac{(1-y)t+(y-x)}{1-x} & x\leq t\leq1\:.\end{array}\right.\]
Now we have:
\begin{itemize}
\item[A1.] $\{\mathsf{P}_1,\mathsf{P}_3\}\in\frac{\mathbb{A}}{\mathcal{G}_{\mathbb{A}}}$:  
	Use Lemma~\ref{salam10} and note that $\varphi(\mathsf{P}_1)=\mathsf{P}_3$.
\item[A2.] $\{\mathsf{P}_2,\mathsf{P}_4\}\in\frac{\mathbb{A}}{\mathcal{G}_{\mathbb{A}}}$: 
	Use Lemma~\ref{salam10} and note that $\varphi(\mathsf{P}_2)=\mathsf{P}_4$.
\item[A3.] $\mathsf{L}_1\cup\mathsf{L}_3\in\frac{\mathbb{A}}{\mathcal{G}_{\mathbb{A}}}$: Using  Lemma~\ref{salam10} 
	we have $\mathcal{G}_\mathbb{A}<0,\frac12>\subseteq\mathcal{G}_\mathbb{A}\mathsf{L}_1\subseteq\mathsf{L}_1\cup\mathsf{L}_3$.
	For $x\in(0,1)$ consider homeomorphism $h:\mathbb{A}\to\mathbb{A}$ with $h<0,t>=<0,f_{\frac{1}{2},x}(t)>$ and
	$h<s,t>=<s,t>$ for $s\neq0$. Then $<0,x>=h<0,\frac{1}{2}>\in \mathcal{G}_\mathbb{A}<0,\frac12>$, thus
	$\mathsf{L}_1\subseteq \mathcal{G}_\mathbb{A}<0,\frac12>$, so
	$\mathsf{L}_1\cup\mathsf{L}_3=\varphi(\mathsf{L}_1)\cup\mathsf{L}_1\subseteq\mathcal{G}_\mathbb{A}<0,\frac12>$ which leads to
	$\mathsf{L}_1\cup\mathsf{L}_3=\mathcal{G}_\mathbb{A}<0,\frac12>\in\frac{\mathbb{A}}{\mathcal{G}_{\mathbb{A}}}$.
\item[A4.] $\mathsf{L}_2\cup\mathsf{L}_4\in\frac{\mathbb{A}}{\mathcal{G}_{\mathbb{A}}}$: Using  Lemma~\ref{salam10}  we have
	$\mathcal{G}_\mathbb{A}<\frac12,0>\subseteq\mathcal{G}_\mathbb{A}\mathsf{L}_4\subseteq\mathsf{L}_2\cup\mathsf{L}_4$.
	For $x\in(0,1)$ consider homeomorphism $h:\mathbb{A}\to\mathbb{A}$ with $h<s,t>=<f_{\frac{1}{2},x}(s),f_{\frac{1}{2},x}(t)>$,
	thus $<x,0>=h<\frac12,0>\in \mathcal{G}_\mathbb{A}<\frac12,0>$ which leads to $\mathsf{L}_4\subseteq\mathcal{G}_\mathbb{A}<\frac12,0>$.
	Thus $\mathsf{L}_2=\varphi(\mathsf{L}_4)\subseteq\varphi(\mathcal{G}_\mathbb{A}<\frac12,0>)=\mathcal{G}_\mathbb{A}<\frac12,0>$
	which leads to $\mathsf{L}_2\cup\mathsf{L}_4=\mathcal{G}_\mathbb{A}<\frac12,0>\in\frac{\mathbb{A}}{\mathcal{G}_{\mathbb{A}}}$.
\item[A5.] $\Delta\setminus\{\mathsf{P}_1,\mathsf{P}_3\}\in\frac{\mathbb{A}}{\mathcal{G}_{\mathbb{A}}}$: Using  Lemma~\ref{salam10} 
	we have $\mathcal{G}_\mathbb{A}<\frac12,\frac{1}{2}>\subseteq\Delta\setminus\{\mathsf{P}_1,\mathsf{P}_3\}$.
	For $x\in(0,1)$ consider homeomorphism $h:\mathbb{A}\to\mathbb{A}$ with $h<s,t>=<f_{\frac{1}{2},x}(s),f_{\frac{1}{2},x}(t)>$
	so $<x,x>=h<\frac12,\frac{1}{2}>\in \mathcal{G}_\mathbb{A}<\frac12,\frac{1}{2}>$ and shows 
	$\Delta\setminus\{\mathsf{P}_1,\mathsf{P}_3\}\subseteq\mathcal{G}_\mathbb{A}<\frac12,\frac{1}{2}>$.
\item[A6.] $((0,1)\times(0,1))\setminus\Delta\in\frac{\mathbb{A}}{\mathcal{G}_{\mathbb{A}}}$: Consider 
	$<a,b>,<c,d>\in ((0,1)\times(0,1))\setminus\Delta$, using (A1), ..., (A5) we have 
	$\mathcal{G}_\mathbb{A}<a,b>\subseteq((0,1)\times(0,1))\setminus\Delta$.
	Consider the following cases:
	\\
	\underline{I.} $b<a$, $d<c$ and $a\leq c$. In this consider homeomorphism 
	$h:\mathbb{A}\to\mathbb{A}$ with $h<s,t>=<f_{a,c}(s),f_{a,c}(t)>$, 
	thus $h<a,b>=<c,f_{a,c}(b)>$ (note
	that $b<a$ thus $f_{a,c}(b)<f_{a,c}(a)=c$). Define $p:\mathbb{A}\to\mathbb{A}$ with:
	{\small
	\[p<s,t>:=\left\{\begin{array}{lc} <s,\dfrac{d}{f_{a,c}(b)}t> & s=c, 0\leq t\leq f_{a,c}(b)\:, \\ & \\ 
	<s,\dfrac{(d-c)t+(f_{a,c}(b)-d)c}{f_{a,c}(b)-c}>& s=c, f_{a,c}(b)\leq t\leq c\:, \\ & \\ <s,t> & {\rm otherwise}\:, \end{array}\right.\]}
	\noindent then $h,p\in\mathcal{G}_\mathbb{A}$ and 
	\[<c,d>=p<c,f_{a,c}(b)>=p(h<a,b>)\in\mathcal{G}_\mathbb{A}<a,b>\:.\]
	\\
	\underline{II.} $b<a$, $d<c$ and $c\leq a$. By case (I) we have $<a,b>\in\mathcal{G}_\mathbb{A}<c,d>$ thus
	there exists $j\in\mathcal{G}_\mathbb{A}$ with $<a,b>=j<c,d>$ so $<c,d>=j^{-1}<a,b>\in\mathcal{G}_\mathbb{A}<a,b>$.
	\\
	\underline{III.} $b<a$ and $d>c$. Choose $e\in(0,c)$ by cases (I) and (II) we have $<c,e>\in\mathcal{G}_\mathbb{A}<a,b>$.
	Define $q:\mathbb{A}\to\mathbb{A}$ with:
	\[q<s,t>:=\left\{\begin{array}{lc} <c,\dfrac{d-1}{e}t+1> & 0\leq t\leq e, s=c\: , \\ & \\ 
	<c,\dfrac{(d-c)t+(e-d)c}{e-c}> & e\leq t\leq c, s=c\: , \\ & \\ 
	<c,\dfrac{c(1-t)}{1-c}> & c\leq t\leq 1, s=c\: , \\ & \\  
	<s,t> & t\neq d \:, \end{array}\right.\]
	then $q\in\mathcal{G}_\mathbb{A}$ and $<c,d>=q<c,e>\in q\mathcal{G}_\mathbb{A}<a,b>=\mathcal{G}_\mathbb{A}<a,b>$.
	\\
	Using cases (I,II, III) we have $((0,1)\times(0,1))\setminus\Delta\subseteq\mathcal{G}_\mathbb{A}<a,b>$ which leads to
	$((0,1)\times(0,1))\setminus\Delta=\mathcal{G}_\mathbb{A}<a,b>\in\frac{\mathbb{A}}{\mathcal{G}_{\mathbb{A}}}$.
\item[O1.] $\{\mathsf{P}_1,\mathsf{P}_3\}\in\frac{\mathbb{O}}{\mathcal{G}_{\mathbb{O}}}$:  
		Use Lemma~\ref{salam20} and note that $\varphi(\mathsf{P}_1)=\mathsf{P}_3$.
\item[O2.] $\{\mathsf{P}_2,\mathsf{P}_4\}\in\frac{\mathbb{O}}{\mathcal{G}_{\mathbb{O}}}$: 
		Use Lemma~\ref{salam20} and note that $\varphi(\mathsf{P}_2)=\mathsf{P}_4$.
\item[O3.] $\mathsf{L}_2\cup\mathsf{L}_4\in\frac{\mathbb{O}}{\mathcal{G}_{\mathbb{O}}}$: 
	By Lemma~\ref{salam20}, 
	$\mathcal{G}_{\mathbb{O}}<\frac{1}{2},0>\subseteq \mathsf{L}_2\cup\mathsf{L}_4$. For $x\in(0,1)$
	consider $h:\mathbb{O}\to\mathbb{O}$ with $h<s,t>=<f_{\frac{1}{2},x}(s),t>$ so
	$<x,0>=h<\frac{1}{2},0>\in\mathcal{G}_{\mathbb{O}}<\frac{1}{2},0>$ and 
	$<x,1>=h(\varphi<\frac{1}{2},0>)\in\mathcal{G}_{\mathbb{O}}<\frac{1}{2},0>$, thus
	$\mathsf{L}_2\cup\mathsf{L}_4\subseteq \mathcal{G}_{\mathbb{O}}<\frac{1}{2},0>$
	which leads to $\mathsf{L}_2\cup\mathsf{L}_4=\mathcal{G}_{\mathbb{O}}<\frac{1}{2},0>\in\frac{\mathbb{O}}{\mathcal{G}_{\mathbb{O}}}$.
\item[O4.] $\mathsf{L}_1\cup\mathsf{L}_3\in\frac{\mathbb{O}}{{\mathcal{G}_{\mathbb O}}}$: 
	By Lemma~\ref{salam20}, 
	$\mathcal{G}_{\mathbb{O}}<0,\frac{1}{2}>\subseteq \mathsf{L}_1\cup\mathsf{L}_3$. For $x\in(0,1)$
	consider $h:\mathbb{O}\to\mathbb{O}$ with $h<s,t>=<s,f_{\frac{1}{2},x}(t)>$ so
	$<0,x>=h<0,\frac{1}{2}>\in\mathcal{G}_{\mathbb{O}}<0,\frac{1}{2}>$ and 
	$<1,x>=h(\varphi<0,\frac{1}{2}>)\in\mathcal{G}_{\mathbb{O}}<0,\frac{1}{2}>$, thus
	$\mathsf{L}_2\cup\mathsf{L}_4\subseteq \mathcal{G}_{\mathbb{O}}<0,\frac{1}{2}>$
	which leads to $\mathsf{L}_2\cup\mathsf{L}_4=\mathcal{G}_{\mathbb{O}}<0,\frac{1}{2}>\in\frac{\mathbb{O}}{\mathcal{G}_{\mathbb{O}}}$.
\item[O5.] $(0,1)\times(0,1)\in\frac{\mathbb{O}}{{\mathcal{G}_{\mathbb O}}}$: Using (O1), (O2), (O3) and (O4) we have
	$\mathcal{G}_\mathbb{O}<\frac{1}{2},\frac{1}{2}>\subseteq(0,1)\times(0,1)$. Choose $<x,y>\in(0,1)\times(0,1)$ and
	define $h:\mathbb{O}\to\mathbb{O}$ with $h<s,t>=<f_{\frac{1}{2},x}(s),f_{\frac{1}{2},y}(t)>$, then
	$<x,y>=h<\frac{1}{2},\frac{1}{2}>$ which shows $(0,1)\times(0,1)\subseteq\mathcal{G}_\mathbb{O}<\frac{1}{2},\frac{1}{2}>$
	and completes the proof.
\end{itemize}
\end{proof}
%%%%%%%%%%%%%%%%%%%%%%%%%%%%%%%%%%%%
\subsection*{Devaney chaos} We say transformation group $(G,X)$ is topological transitive if for all opene subset $U,V$ of $X$ we have
$U\cap GV\neq\varnothing$. In transformation group $(G,X)$ we say $x\in X$ is a periodic point if $st(x):=\{g\in G:gx=x\}$
is a finite index subgroup of $G$. Transformation group $(G,X)$ is Devaney chaotic if it is topological transitive and the collection of
its periodic points is dense in $X$ \cite{sen}. Moreover in $(G,X)$ we say $x\in X$ is an almost periodic point if
$\overline{Gx}$ is a minimal subset of $X$ (i.e., it is closed invariant subset of $(G,X)$
without any proper subset which is closed invariant subset of $(G,X)$) \cite{ellis}. All periodic points of $(G,X)$
are almost periodic. Using the following theorem, transformation groups $(\mathcal{G}_\mathbb{A},\mathbb{A})$,
$(\mathcal{G}_\mathbb{O},\mathbb{O})$ and $(\mathcal{G}_\mathbb{U},\mathbb{U})$ are not Devaney chaotic.
%%%%%%%%%%%%%%%%%%%%%%%%%%%%%%%%%%%%
\begin{lemma}\label{salam60}
The transformation groups $(\mathcal{G}_\mathbb{A},\mathbb{A})$ and
$(\mathcal{G}_\mathbb{O},\mathbb{O})$ are not topological transitive, however
$(\mathcal{G}_\mathbb{U},\mathbb{U})$ is topological transitive.
\end{lemma}
%%%%%%%%%%%%%%%%%%%%%%%%%%%%%%%%%%%%
\begin{proof}
For $X=\mathbb{A},\mathbb{O}$ the sets $U:=(0,1)\times(0,1)$ and
$V:=\mathsf{L}_1\cup\mathsf{L}_3$ are open subsets of $X$
and by Theorem~\ref{salam30} we have $\mathcal{G}_XU\cap V=U\cap V=\varnothing$ thus $(\mathcal{G}_X,X)$ is not topological transitive.
\end{proof}
%%%%%%%%%%%%%%%%%%%%%%%%%%%%%%%%%%%%
\begin{theorem}
For $X=\mathbb{A},\mathbb{O}$ transformation group
$(G,X)$ is not topological transitive, in particular it is not
Devaney chaotic.
\end{theorem}
%%%%%%%%%%%%%%%%%%%%%%%%%%%%%%%%%%%%
\begin{proof}
By Lemma~\ref{salam60} there are opene subsets $U,V$ of $X$  with
$\mathcal{G}_XU\cap V=\varnothing$. Using $GU\subseteq\mathcal{G}_XU$
we have $GU\cap V=\varnothing$ which completes the proof.
\end{proof}
%%%%%%%%%%%%%%%%%%%%%%%%%%%%%%%%%%%%
\begin{note}
In transformation group $(\mathcal{G}_\mathbb{A},\mathbb{A})$ (resp. $(\mathcal{G}_\mathbb{O},\mathbb{O})$),
$\mathfrak{x}$ is an almost periodic point if and only if $\mathfrak{x}$ is a periodic point. Also $\{\mathsf{P}_i:1\leq i\leq4\}$ is the collection
of all its periodic points. Moreover $(\mathcal{G}_\mathbb{U},\mathbb{U})$ does not have any periodic point, but
$\{<s,t>\in\mathbb{U}:\{s,t\}\cap\{0,1\}\neq\varnothing\}$ is the collection of its almost periodic points.
\end{note}
%%%%%%%%%%%%%%%%%%%%%%%%%%%%%%%%%%%%
%%%%%%%%%%%%%%%%%%%%%%%%%%%%%%%%%%%%
\section{Computing $P_h(\mathbb{A})$, $P_h(\mathbb{O})$ and $P_h(\mathbb{U})$}
\noindent Now we are ready to find out $P_h(\mathbb{A})$, $P_h(\mathbb{O})$ and $P_h(\mathbb{U})$. We show 
$P_h(\mathbb{A})=\{n:n\geq5\}\cup\{+\infty\}$, $P_h(\mathbb{O})=\{n:n\geq4\}\cup\{+\infty\}$ and $P_h(\mathbb{U})=\{n:n\geq1\}\cup\{+\infty\}$.
%%%%%%%%%%%%%%%%%%%%%%%%%%%%%%%%%%%%
\begin{theorem}\label{salam40}
$h(\mathcal{G}_\mathbb{A},\mathbb{A})=5$, $h(\mathcal{G}_\mathbb{O},\mathbb{O})=4$, $h(\mathcal{G}_\mathbb{U},\mathbb{U})=1$.
\end{theorem}
%%%%%%%%%%%%%%%%%%%%%%%%%%%%%%%%%%%%
\begin{proof}
Use Theorem~\ref{salam30}.
\end{proof}
%%%%%%%%%%%%%%%%%%%%%%%%%%%%%%%%%%%%
\begin{theorem}
$P_h(\mathbb{A})=\{n:n\geq5\}\cup\{+\infty\}$, $P_h(\mathbb{O})=\{n:n\geq4\}\cup\{+\infty\}$, $P_h(\mathbb{U})=\{n:n\geq1\}\cup\{+\infty\}$.
\end{theorem}
%%%%%%%%%%%%%%%%%%%%%%%%%%%%%%%%%%%%
\begin{proof}
{\it Computing $P_h(\mathbb{A})$.} By Theorem~\ref{salam40}, it's evedent that $5\in P_h(\mathbb{A})\subseteq\{n:n\geq5\}\cup\{+\infty\}$.
For $n\geq1$ choose 
$t_1,\ldots,t_n\in(0,1)$ with $\frac12=t_1<\cdots,t_n$ and let
\begin{eqnarray*}
\mathcal{H}_{\mathbb A} & := & \{f\in\mathcal{G}_{\mathbb A}:f(\mathsf{P}_1)=\mathsf{P}_1\} \\
\mathcal{K}_0 & := & \{f\in\mathcal{G}_{\mathbb A}:f<0,t_1>=<0,t_1>,\ldots,f<0,t_n>=<0,t_n>\}(\subseteq\mathcal{H}_{\mathbb A}) \\
\mathcal{K}_1 & := & \{f\in\mathcal{K}_0:f(\{<0,\frac1j>:j\geq2\}\cup\{<0,\frac12-\frac1j>:j\geq3\})= \\
	& & \SP\SP\SP\SP\SP\SP \{<0,\frac1j>:j\geq2\}\cup\{<0,\frac12-\frac1j>:j\geq3\}\} \\
\mathcal{K}_2 & := & \{f\in\mathcal{G}_{\mathbb A}:f(\{<0,\frac12>,<1,\frac12>\})=\{<0,\frac12>,<1,\frac12>\}\} \\
\mathcal{K}_3 & := & \{f\in\mathcal{G}_{\mathbb A}:f(\{<i,\frac1j>:j\geq2,i=0,1\}\cup\{<i,1-\frac1j>:j\geq2,i=0,1\})= \\
	& & \SP\SP\SP\SP\SP\SP \{<i,\frac1j>:j\geq2,i=0,1\}\cup\{<i,1-\frac1j>:j\geq2,i=0,1\}\} 
\end{eqnarray*}
Then $\mathcal{H}_{\mathbb A}$
is a proper normal subgroup of $\mathcal{G}_{\mathbb A}$ with index 2 and $\mathcal{G}_{\mathbb A}=\mathcal{H}_{\mathbb A}\cup
\varphi\mathcal{H}_{\mathbb A}$ (where $\varphi<s,t>=<1-s,1-t>$). Moreover using a similar method described in 
Theorem~\ref{salam30} we have:
\begin{eqnarray*}
\frac{\mathbb{A}}{\mathcal{H}_{\mathbb{A}}} & = & 
	\{\{\mathsf{P}_1\},\{\mathsf{P}_2\},\{\mathsf{P}_3\},\{\mathsf{P}_4\},\mathsf{L}_1,\mathsf{L}_3,
	\mathsf{L}_2\cup\mathsf{L}_4,\Delta\setminus\{\mathsf{P}_1,\mathsf{P}_3\},((0,1)\times(0,1))\setminus\Delta\} \\
\frac{\mathbb{A}}{\mathcal{K}_0} & = & (\frac{\mathbb{A}}{\mathcal{H}_{\mathbb{A}}}\setminus\{\mathsf{L}_1\})\cup
	\{\{<0,t_1>\},\ldots.\{<0,t_n>\}, \\
	& & \SP \{0\}\times(0,t_1), \{0\}\times(t_1,t_2),\ldots,\{0\}\times(t_{n-1},t_n),\{0\}\times(t_n,1)\} \\
\frac{\mathbb{A}}{\mathcal{K}_1} & = & (\frac{\mathbb{A}}{\mathcal{K}_0}\setminus\{\{0\}\times(0,t_1)\})\cup\{\{<0,\frac1j>:j\geq2\}\cup\{<0,\frac12-\frac1j>:j\geq3\}, \\
	& & \SP \{0\}\times((0,t_1)\setminus\{\frac1j:j\geq2\}\cup\{\frac12-\frac1j:j\geq3\})\} \\
\frac{\mathbb{A}}{\mathcal{K}_2} & = & (\frac{\mathbb{A}}{\mathcal{G}_{\mathbb A}}\setminus\{\mathsf{L}_1\cup\mathsf{L}_3\})\cup \{\{<0,\frac12>,<1,\frac12>\}, \\
	& & \SP (\{0\}\times(0,\frac12))\cup(\{1\}\times(\frac12,1)),(\{0\}\times(\frac12,1))\cup(\{1\}\times(0,\frac12))\} \\
\frac{\mathbb{A}}{\mathcal{K}_3} & = & (\frac{\mathbb{A}}{\mathcal{G}_{\mathbb A}}\setminus\{\mathsf{L}_1\cup\mathsf{L}_3\})\cup \\
	& & \SP \{\{<i,\frac1j>:j\geq2,i=0,1\}\cup\{<i,1-\frac1j>:j\geq2,i=0,1\}, \\
	& & \SP (\mathsf{L}_1\cup\mathsf{L}_3)\setminus(\{<i,\frac1j>:j\geq2,i=0,1\}\cup\{<i,1-\frac1j>:j\geq2,i=0,1\})\} 
\end{eqnarray*}
which leads to $h(\mathcal{H}_{\mathbb A},\mathbb{A})=8$, $h(\mathcal{K}_0,\mathbb{A})=8+2n$, 
$h(\mathcal{K}_1,\mathbb{A})=8+2n+1$, 
$h(\mathcal{K}_2,\mathbb{A})=7$, $h(\mathcal{K}_3,\mathbb{A})=6$, $h(\{id_{\mathbb A}\},\mathbb{A})=+\infty$.
Hence $P_h(\mathbb{A})=\{n:n\geq5\}\cup\{+\infty\}$.
\\
{\it Computing $P_h(\mathbb{O})$.} By Theorem~\ref{salam40}, it's evedent that $4\in P_h(\mathbb{O})\subseteq\{n:n\geq4\}\cup\{+\infty\}$.
For $n\geq1$ choose 
$t_1,\ldots,t_n\in(0,1)$ with $\frac12=t_1<\cdots,t_n$ and let
\begin{eqnarray*}
\mathcal{H}_{\mathbb O} & := & \{f\in\mathcal{G}_{\mathbb O}:f(\mathsf{P}_1)=\mathsf{P}_1\} \\
\mathcal{J}_0 & := & \{f\in\mathcal{G}_{\mathbb O}:f<0,t_1>=<0,t_1>,\ldots,f<0,t_n>=<0,t_n>\}(\subseteq\mathcal{H}_{\mathbb O}) \\
\mathcal{J}_1 & := & \{f\in\mathcal{J}_0:f(\{<0,\frac1j>:j\geq2\}\cup\{<0,\frac12-\frac1j>:j\geq3\})= \\
	& & \SP\SP\SP\SP\SP\SP \{<0,\frac1j>:j\geq2\}\cup\{<0,\frac12-\frac1j>:j\geq3\}\} \\
\mathcal{J}_2 & := & \{f\in\mathcal{G}_{\mathbb O}:f(\{<0,\frac12>,<1,\frac12>\})=\{<0,\frac12>,<1,\frac12>\}\} \\
\mathcal{J}_3 & := & \{f\in\mathcal{G}_{\mathbb O}:f(\{<i,\frac1j>:j\geq2,i=0,1\}\cup\{<i,1-\frac1j>:j\geq2,i=0,1\})= \\
	& & \SP\SP\SP\SP\SP\SP \{<i,\frac1j>:j\geq2,i=0,1\}\cup\{<i,1-\frac1j>:j\geq2,i=0,1\}\} \\
\mathcal{J}_4 & := & \{f\in\mathcal{G}_{\mathbb O}:f(\{\frac12\}\times(0,1))=\{\frac12\}\times(0,1)\}
\end{eqnarray*}
One can verify $h(\mathcal{H}_{\mathbb O},\mathbb{O})=8$, $h(\mathcal{J}_0,\mathbb{O})=8+2n$, 
$h(\mathcal{J}_1,\mathbb{O})=8+2n+1$, 
$h(\mathcal{J}_2,\mathbb{O})=6$, $h(\mathcal{J}_3,\mathbb{A})=5$, $h(\mathcal{J}_4,\mathbb{O})=7$, $h(\{id_{\mathbb O}\},\mathbb{O})=+\infty$.
Hence $P_h(\mathbb{A})=\{n:n\geq4\}\cup\{+\infty\}$.
\\
{\it Computing $P_h(\mathbb{U})$.} By Theorem~\ref{salam40}, it's evedent that $1\in P_h(\mathbb{U})\subseteq\{n:n\geq1\}\cup\{+\infty\}$.
For $n\geq1$ choose distinct
$\mathfrak{x}_1,\ldots,\mathfrak{x}_n\in(0,1)\times(0,1)$, so $h(\{f\in\mathcal{G}_\mathbb{U}:f(\mathfrak{x}_1)=\mathfrak{x}_1,\ldots,f(\mathfrak{x}_n)=\mathfrak{x}_n\},\mathbb{U})=n+1$ and $h(\{id_{\mathbb U}\},\mathbb{U})=+\infty$ which completes the proof.
\end{proof}
%%%%%%%%%%%%%%%%%%%%%%%%%%%%%%%%%%%%
\section*{Acknowledgement}
\noindent The authors are grateful to the research division of the University of Tehran
 for the grant which supported this research.
%%%%%%%%%%%%%%%%%%%%%%%%%%%%%%%%%%%%
%%%%%%%%%%%%%%%%%%%%%%%%%%%%%%%%%%%%
%%%%%%%%%%%%%%%%%%%%%%%%%%%%%%%%%%%%

\noindent {\small 
{\bf Fatemah Ayatollah Zadeh Shirazi},
Faculty of Mathematics, Statistics and Computer Science,
College of Science, University of Tehran,
Enghelab Ave., Tehran, Iran
\\
({\it e-mail}: fatemah@khayam.ut.ac.ir)
\\
{\bf Fatemeh Ebrahimifar},
Faculty of Mathematics, Statistics and Computer Science,
College of Science, University of Tehran ,
Enghelab Ave., Tehran, Iran
\\
({\it e-mail}: ebrahimifar64@ut.ac.ir)
\\
{\bf Reza Yaghmaeian},
Faculty of Mathematics, Statistics and Computer Science,
College of Science, University of Tehran,
Enghelab Ave., Tehran, Iran
\\
({\it e-mail}: rezayaghma@yahoo.com)
\\
{\bf Hamed Yahyaoghli},
Department of Mathematics, Tarbiat Modares University,
Modiriat Bridge, Tehran, Iran
\\
({\it e-mail}: yahyaoghli@gmail.com)}
%%%%%%%%%%%%%%%%%%%%%%%%%%%%%%%%%%%%
%\[\underline{\SP\SP\SP\SP\SP\SP\SP\SP\SP\SP\SP\SP\SP\SP\SP\SP}\]

\end{document}